\documentclass[twocolumn]{autart}    

\usepackage{graphicx}          

\usepackage[utf8]{inputenc}
\usepackage{amsmath, amsfonts, amssymb}

\usepackage{url}
\usepackage{subfiles}
\usepackage{xcolor}
\usepackage[font={footnotesize,sf},labelfont=bf]{caption}
\usepackage{subcaption}
\usepackage{tikz}
\usetikzlibrary{math,arrows.meta}
\usetikzlibrary{
	tikzmark, quotes, angles, backgrounds, scopes, calc,
decorations.pathreplacing, shapes, math, shapes.geometric}
\definecolor{target}{RGB}{11,130,12}
\usepackage[normalem]{ulem}
\usepackage{cite}

\usepackage{enumitem}

\newtheorem{theorem}{Theorem}
\newtheorem{lemma}{Lemma}
\newtheorem{corr}{Corollary}
\newtheorem{answer}{Answer}

\newtheorem{definition}{Definition}
\newtheorem{question}{Question}

\newenvironment{proof}{%
    \begin{pf}%
}{%
    \end{pf}%
    \ignorespacesafterend%
}\newcommand{\rvec}{\mathbf{r}}
\newcommand{\rvechat}{\hat{\mathbf{r}}}
\newcommand{\x}{{\bf{x}}}
\newcommand{\vvec}{{\bf{v}}}
\newcommand{\y}{{\bf{y}}}
\newcommand{\z}{{\bf{z}}}
\newcommand{\reals}{\mathbb{R}}

\newcommand{\T}{^{\mbox\footnotesize\sf T}}
\newcommand{\xf}{\x_\mathcal{C}}
\DeclareMathOperator*{\argmin}{arg\,min}
\newcommand{\prettyinf}{\mathop{\mathrm{inf}\vphantom{\mathrm{sup}}}}
\newcommand{\nut}{\nu_{\rm{true}}}

\begin{document}

\begin{frontmatter}

\title{One Apollonius Circle is Enough for Many Pursuit-Evasion Games} 

\thanks[footnoteinfo]{The authors are listed alphabetically. The first two authors contributed equally and should be considered co-first-authors. This paper was not presented at any IFAC 
meeting. Corresponding author M. Dorothy. Email Address: michael.r.dorothy.civ@army.mil.
We gratefully acknowledge the support of ARL grant ARL DCIST CRA W911NF-17-2-0181. The views expressed in this paper are those of the authors and do not reflect the official policy or position of the United States Government, Department of Defense, or its components.}


\author[ARL]{Michael Dorothy}
\author[UNCC]{Dipankar Maity}
\author[GMU]{Daigo Shishika}
\author[AFRL]{Alexander Von Moll}


\address[ARL]{Computational and Information Sciences Directorate, Army Research Laboratory, APG, MD}
\address[UNCC]{Department of Electrical and Computer Engineering, University of North Carolina, Charlotte, NC}
\address[GMU]{Department of Mechanical Engineering, George Mason University, Fairfax, VA}
\address[AFRL]{Control Science Center, Air Force Research Laboratory, WPAFB, OH}


\begin{abstract}                          
This paper investigates obstacle-free simple motion pursuit-evasion problems where the pursuer is faster and game termination is point capture.
It is well known that the interior of the Apollonius Circle (AC) is the evader's dominance region, however, it was unclear whether the evader could reach outside the initial AC without being captured. 
We construct a pursuit strategy that guarantees the capture of an evader within an arbitrarily close neighborhood of the initial AC.
The pursuer strategy is derived by reformulating the game into a nonlinear control problem, and the guarantee holds against any admissible evader strategy.
Our result implies that the evader can freely select the capture location, but only inside the initial AC.
Therefore, a class of problems, including those where the payoff is determined solely based on the location of capture, are now trivial.

\end{abstract}

\end{frontmatter}

\section{Introduction}
In his seminal work, Rufus Isaacs \cite{isaacs1965differential}, among other things, formulated pursuit-evasion scenarios as differential games.
In the decades since, applications of games of this type have grown to include missile defense~\cite{perelman2011cooperative}, football tackling strategy, target/coastline guarding, and more~\cite{shishika2020review}.
Even in Isaacs' work, it was clear that particular circles, associated with Apollonius of Perga and known for some two millenia prior, were an important tool in solving these problems. An Apollonius Circle (AC) is defined as the set of points in a plane that have a specified ratio of distances to two fixed points, known as foci\footnote{For a proof of why this locus is a circle see, e.g., \cite[Appendix 2]{weintraub2020optimalguidance}.}.
When two agents with simple motion (single-integrator kinematics) are able to freely move on a plane, an AC can be drawn at a moment in time with the two foci being the current location of the agents and the ratio of distances corresponding to the ratio of the two agents' maximum velocities. It is well-known that with a faster pursuer, the AC encloses a set of points (``evader dominance region'') that the evader can safely reach before the pursuer can capture it. 

Isaacs noticed the relevance of this result to a class of pursuit evasion games that involve targets. The evader wins a target defense game by reaching the target before being captured by the pursuer, whereas the pursuer wins the game by ensuring that capture occurs first~\cite{venkatesan2014target}. Isaacs' initial result was concerning a convex target, but the target guarding problem has been extended to other variants with targets having various shapes, inspired by and modeling different practical scenarios. 
Other works were motivated by the similar problem of capturing an ``escaping'' agent, e.g., lion-man problem, the target surface enclosed the players~\cite{garcia2019optimal}.

In all of these games, if the AC intersects the target, then the evader is guaranteed to be able to reach the target without being captured by the pursuer.
The evader can simply aim at a point on the target inside of the AC and move with maximum velocity. However, some subtlety exists in the converse: ``can the pursuer ensure capture if the AC does not initially intersect the target?'' The pursuer does not have a fixed point in space that it can reach in order to win; it must end up at the same location where the evader ends up - and the evader gets a choice in the matter!

This subtlety was the motivation for many of the works published in this field, and specific differential games were formulated to treat each problem separately~\cite{lee2021guarding,barron2018reach,bansal2017hamilton}. Even though the standard methods provide strong guarantees in the region of the state space in which the value function is differentiable, special treatments are often required along the singular surfaces.
Even the dispersal surface (see 6.1 of~\cite{isaacs1965differential}), which is considered to be relatively benign (in comparison with e.g., focal, equivocal, switching, and transition surfaces), has been demonstrated to yield some interesting behaviors.
For example, in the two-pursuer one-evader game of \textit{minmax capture time} with simple motion, faster pursuers, and point capture it was shown in~\cite{pachter2019singular} that the canonical pursuer strategy~\cite{isaacs1965differential} results in chattering in the vicinity of the dispersal surface if the evader stands still.
Ultimately, the evader cannot increase its capture time beyond the Value of the game in this attempt to exploit the singularity, however, its ability to induce chattering in the pursuers' headings is undesirable.
Additionally, there is still the issue of what Isaacs referred to as a perpetuated dilemma in which even if the players happen to select the same equilibrium while on the dispersal surface, the state of the system may \textit{remain} in a singular configuration.
In the case of Isaacs' wall pursuit game, a resolution for the perpetuated dilemma has been proposed in~\cite{milutinovic2021rate} which simultaneously addresses the pursuers' chattering.
For target guarding, however, the consequences of non-uniqueness may be more dire.
As we will see in Section~\ref{sec:ex2}, this type of chattering in a pursuer's heading could lead to growth of the AC, even to the point where the evader can win.

We provide a  pursuer control that ensures capture of an evader within an arbitrary close neighborhood of the initial AC. The proposed pursuit strategy provides a satisficing control to solve a broad class of pursuit-evasion games (such as target guarding, reach-avoid) studied in the literature, and is, in many cases, approximately equal to the traditional equilibrium strategies.

The first significance of the result provided in this paper is that the proposed controller solves a wide class of problems without any special treatment.
The second significance of the presented work is that the proposed controller is a continuous state-feedback function.
As we will see, many difficulties come from large changes in heading (i.e., `chattering') that traditional methods may have near singular surfaces.
Our control law is smoothly defined everywhere on the state-space, which prohibits the evader from exploiting any singularity.

\subsection{Motivating Example 1: Pure pursuit allows the AC to bulge out}\label{sec:ex1}
The statement is often made that capture of the evader within the AC (at initial time) is guaranteed without explicitly proposing a pursuit strategy which does so (see, e.g.,~\cite{yan2017defense, garcia2019optimal}).
This statement should be made more precisely: if the pursuer knows a specific trajectory that the evader is going to take or has access to the evader's control at each timestep, then the pursuer can guarantee capture within the original AC~\cite{yan2021cooperative}.

Worse, still, is when an equilibrium strategy associated with capture time is utilized in a game over capture \textit{location} (e.g., implementing pure pursuit in the evader's win region of a circular escape game~\cite{yan2017defense}).
Against pure pursuit (pursuer always moves at maximum velocity directly toward the evader's current position), an evader's reachable set is a strict superset of the AC~\cite{makkapati2019optimal} -- this is of utmost importance in games wherein capture \textit{location} is pertinent (as shown in Fig~\ref{fig:apollonius_leakage}).

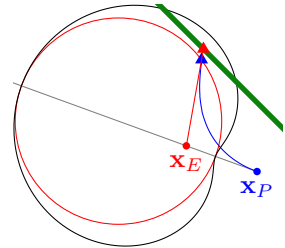
\begin{figure}[htpb]
	\centering
	\begin{tikzpicture}[scale=1]

	\tikzmath{
		\a = 0.70;
		\RHO = \a / (1 - \a^2);
		\th = 80;
		\d = 1;
		\ei = \a * \RHO * cos(\th) + \RHO * sqrt(\a^2 * (cos(\th))^2 - \a^2 + 1);
		\thetapp = 90 - \th;
		\w = (1 - sin(\thetapp)) / cos(\thetapp);
		\tf = (1 + \a * sin(\thetapp)) / (1 - \a^2) * \d;
		\ec = \a * \tf;
		\xpo = \d * cos(\thetapp);
		\globalrotation = 160;
	}

    \coordinate (E) at (0, 0);
	\path (E) -- ++(\globalrotation:-\d) coordinate (P);
	\path (E) -- ++(-\th + \globalrotation:\ei) coordinate (I);
	\path [color=red] (E) -- ++(-\th + \globalrotation:\ec) coordinate (C);
	\path (E) -- ++(\globalrotation:{\a^2/(1-\a^2)}) coordinate (O);
    \draw [color=red] (O) circle [radius=\RHO];

	\tikzmath{
		\pprotation = \globalrotation - 90;
		\xpf = 0.004 * \xpo;
	}
	\draw [opacity=1,blue, domain=\xpf:\xpo, samples=700, smooth] 
		plot (
			{cos(\thetapp + \pprotation) * (\d * \a / (1 - \a^2) * (1 + \a * sin(\thetapp)) + \d * 0.5 * cos(\thetapp) * (1 / (1 + \a) * \w * (\x / \d / cos(\thetapp))^(1 + \a) - 1 / (1 - \a) / \w * (\x / \d / cos(\thetapp))^(1 - \a))) + sin(\thetapp + \pprotation) * \x},
			{sin(\thetapp + \pprotation) * (\d * \a / (1 - \a^2) * (1 + \a * sin(\thetapp)) + \d * 0.5 * cos(\thetapp) * (1 / (1 + \a) * \w * (\x / \d / cos(\thetapp))^(1 + \a) - 1 / (1 - \a) / \w * (\x / \d / cos(\thetapp))^(1 - \a))) - cos(\thetapp + \pprotation) * \x}
			);

	\node [regular polygon, regular polygon sides=3, inner sep=1pt, fill, color=blue] at ({cos(\thetapp + \pprotation) * (\d * \a / (1 - \a^2) * (1 + \a * sin(\thetapp)) + \d * 0.5 * cos(\thetapp) * (1 / (1 + \a) * \w * (\xpf / \d / cos(\thetapp))^(1 + \a) - 1 / (1 - \a) / \w * (\xpf / \d / cos(\thetapp))^(1 - \a)))},
		{sin(\thetapp + \pprotation) * (\d * \a / (1 - \a^2) * (1 + \a * sin(\thetapp)) + \d * 0.5 * cos(\thetapp) * (1 / (1 + \a) * \w * (\xpf / \d / cos(\thetapp))^(1 + \a) - 1 / (1 - \a) / \w * (\xpf / \d / cos(\thetapp))^(1 - \a)))}) {};

	\node [circle, fill, blue, inner sep=1pt] at (P) {};
	\node [circle, fill, red, inner sep=1pt] at (E) {};
    \node [blue] at (P) [below,yshift=0pt] {$\x_P$};
	\node [red] at (E) [below,yshift=0pt] {$\x_E$};


	\draw [color=black, fill opacity=0.15, domain=0:360, samples=100, smooth]
		plot (xy polar cs:
		angle={\x + \globalrotation},
		radius={\a * \d * (1 + \a * sin(90 - \x)) / (1 - \a^2)});
	\clip (current bounding box.south west) rectangle (current bounding box.north east);
	\begin{scope}[on background layer]
		\clip (current bounding box.south west) rectangle (current bounding box.north east);
		\draw [thin, color=gray] (P) -- ($(P)!5!(O)$);
    \end{scope}

	\path (E) -- ($(E)!.85!(C)$) coordinate (T);
	\tikzmath{\tilt = -45;}
	\draw [line width=2pt, color=target] (T) -- ++(\tilt:3);
	\draw [line width=2pt, color=target] (T) -- ++(\tilt:-3);
	\node [regular polygon, regular polygon sides=3, inner sep=1pt, fill, color=red] at (T) {};
	\draw [red] (E) -- (T);

\end{tikzpicture}
	\caption{If the pursuer implements pure pursuit, the evader can reach the target (green) even though the AC (red) does not intersect; the black curve represents locus of pure pursuit capture points for constant evader heading.}
	\label{fig:apollonius_leakage}
\end{figure}

\subsection{Motivating Example 2: Two target problem leading to oscillation and evasion}\label{sec:ex2}

Appendix~\ref{sec:twoGoal} details a target-guarding problem with two simple linear target regions (depicted in Figure~\ref{fig:ex2}). A naive pursuer strategy could consider which target is more `dangerous' and focus on defending that target. The expected behavior from traditional methods is shown in Figure~\ref{fig:ex2}(a), where the game proceeds in a stable fashion toward one of the two targets, and it appears as though neither agent ever has an incentive to choose any other strategy.

\begin{figure}[ht]
    \centering
    \includegraphics[width=0.99\linewidth]{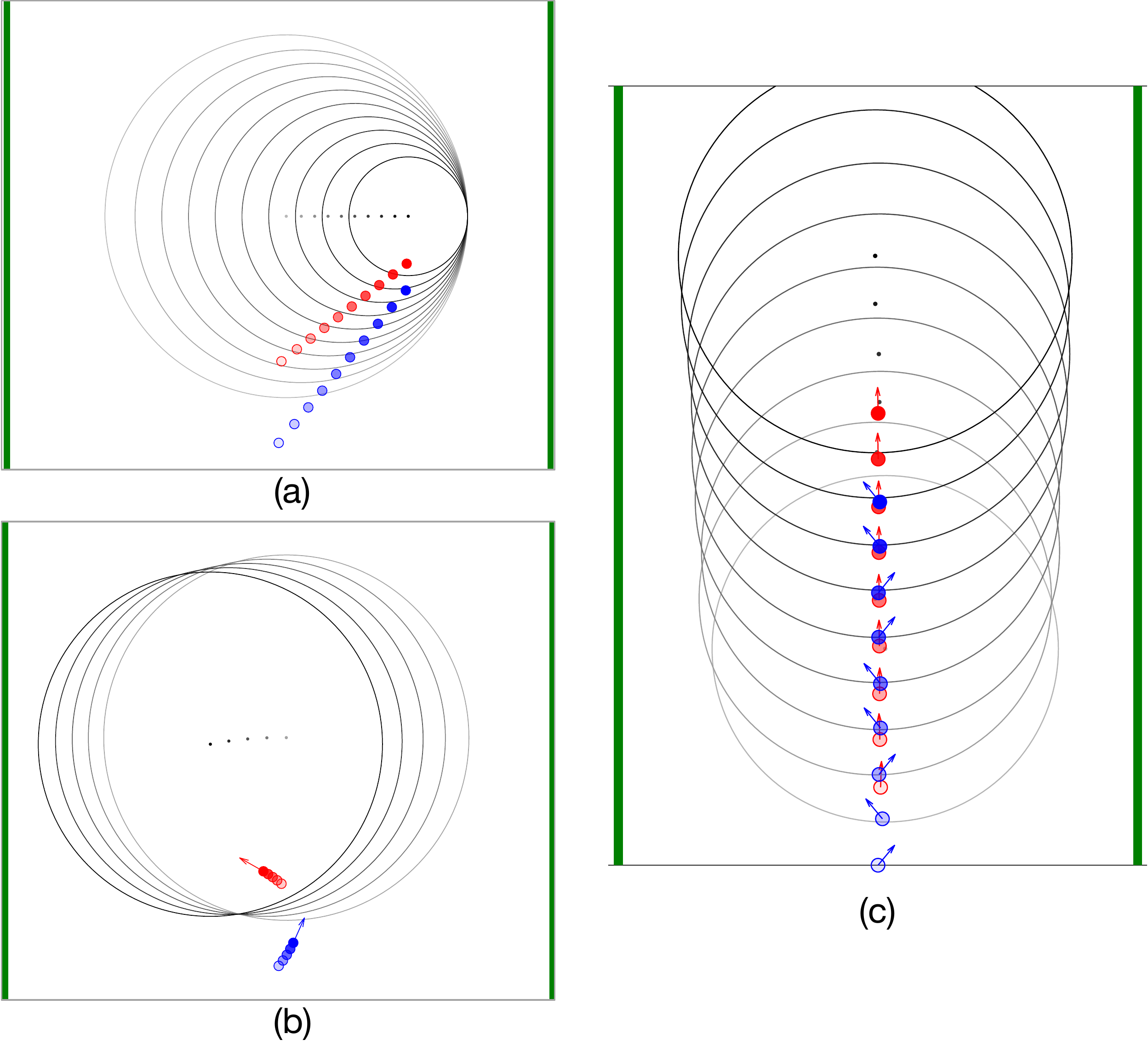}
    \caption{Simulation of a two-target problem. All simulations start from the same initial configuration. (a) Both players employ the equilibrium strategies found from differential game formulation. (b) The evader moves towards left target while the pursuer protects the right target. (c) The evader uses a pure-evasion strategy. }
    \label{fig:ex2}
\end{figure}

Now, consider in Figure~\ref{fig:ex2}(b) the case where the evader heads toward the other target. Like Motivating Example 1, this results in the time-dependent AC moving outside of the original AC. Nevertheless, it can be believed that if the pursuer is focused on the more dangerous target, the evader only loses out by choosing otherwise and that the defender can always switch before the second target becomes more dangerous than the first.

Appendix~\ref{sec:twoGoal} details this naive pursuer strategy based off of traditional methods. It then compares it against a \textit{completely different} admissible evader strategy (go straight up) than any of the strategies identified by traditional methods. The result (depicted in Figure~\ref{fig:ex2}(c)) is that the evader's strategy exploits a singular surface, causes the pursuer's control to chatter, and ultimately results in the evader winning the game.

Therefore, one must be careful when reasoning about equilibrium behaviors in the vicinity of singularities.
In the case of this dispersal surface, it is tempting to believe the evader should only choose to aim towards the left or right goal.
However, from this example, we see that it is imperative to consider the opponent's entire strategy space.
The pursuer must, in general, employ a strategy which guarantees capture against any evader strategy (which the naive, bang-bang pursuer strategy does not).

\section{Main Results}

Define the state of the game at time $t$ as $\x_E(t), \x_P(t) \in \mathbb{R}^2$, i.e., each agent's position in a fixed Cartesian frame.
The agents control their instantaneous velocity: $\dot{\x}_E = \vvec_E$, $\dot{\x}_P = \vvec_P$ which are bounded by $\lVert \vvec_E \rVert \le \nu$ and $\lVert \vvec_P \rVert \le 1$, where $\nu < 1$.  
The game starts at time $t_0$, and terminates when capture occurs: $\x_P = \x_E$.
While the AC was originally constructed as a circle\footnote{
In three dimensional space it becomes a sphere and the presented analysis holds for 3D systems as well.
}~\cite{weintraub2020optimalguidance}, for notational simplification, we define it to be the closed disc 
\begin{equation} \label{eq:ap_circle}
    \mathcal{A}(t) \triangleq \left\{\x : \lVert \x - \x_\mathcal{A}(t) \rVert \leq R_{\mathcal{A}}(t) \right\},
\end{equation}
where $\x_\mathcal{A}(t)= \alpha \x_E(t)-\beta \x_P(t)$, $R_{\mathcal{A}}(t)=\gamma \|\rvec(t)\|$, $\rvec(t)=\x_E(t)-\x_P(t)$,  and $\beta=\nu^2(1-\nu^2)^{-1} = \nu \gamma = \nu^2\alpha$. 
Note that, since $\nu\in[0,1)$, $\alpha\ge \gamma\ge \beta\ge 0$. The traditional AC is the boundary of $\mathcal{A}(t)$, and may be denoted as $\partial \mathcal{A}(t)$.

We claim that there exists a (feedback) pursuer strategy which ensures capture arbitrarily close to $\mathcal{A}(t_0)$ regardless of the strategy of the attacker.
To prove our claim, we construct a fixed disc $\mathcal{C}$ centered at $\xf=\x_\mathcal{A}(t_0)$ with radius $R_\mathcal{C} = R_\mathcal{A}(t_0) + \delta$, where $\delta >0$, i.e.,
\begin{align} \label{eq:capture_region}
    \mathcal C = \{\x : \|\x- \x_\mathcal{C}\| \le R_\mathcal{C}\}. 
\end{align}
We will show that $\mathcal{A}(t)$ will be contained within the interior of $\mathcal{C}$ (denoted $\mathcal{C}^\circ$) for all $t\ge t_0$ and that the distance between the pursuer and the attacker will go to zero in a finite time.

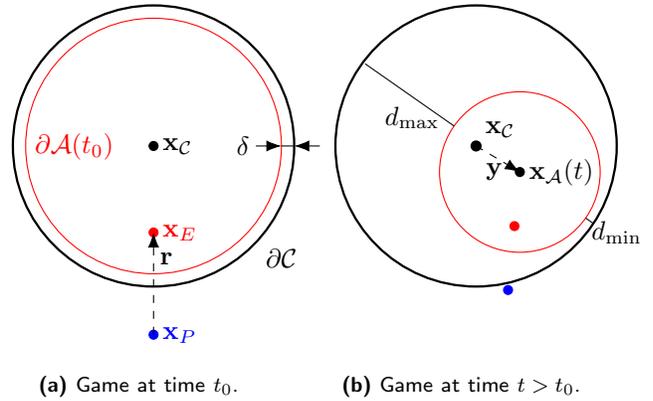
\begin{figure}[h]
\begin{subfigure}{0.4\linewidth}
\begin{tikzpicture}[scale=1.7]
    \tikzmath{\yD = 0; \yA = .8; \tnu=.677; \tdelta=.1;
    \talpha=1/(1-\tnu^2); \tgamma=\talpha*\tnu; \tbeta=\tgamma*\tnu;
    \yC=\talpha*\yA-\tbeta*\yD; \tr=\yA-\yD; \tRC=\tgamma*\tr; \tRF=\tRC+\tdelta;}

	\draw[red,fill] (0,\yA) circle[radius=1pt] node [anchor=west] {$\x_E$};
	\draw[blue,fill](0,\yD) circle[radius=1pt] node [anchor=west] {$\x_P$};
	\draw[black,fill](0,\yC) circle[radius=1pt];
	\draw[-{Latex[length=2mm]},dashed] (0,\yD)--(0,\yA);
	
	\draw[red](0,\yC) circle[radius=\tRC];
	\draw[line width=0.3mm, black] (0,\yC) circle[radius=\tRF];
	
	\node(C) at (\tRC,.6){$\partial\mathcal{C}$};
	\node(A)[red] at (-\tRC,\yC) [anchor=west] {$\partial\mathcal{A}(t_0)$};
	\node(xF) at (0,\yC) [anchor=west] {$\x_\mathcal{C}$};
	\node(delta) at (\tRC-.3,\yC){$\delta$};
	\draw[] (\tRC,\yC)--(\tRF,\yC);
	\draw[-{Latex[length=2mm]}] (\tRC-.2,\yC)--(\tRC,\yC);
	\draw[-{Latex[length=2mm]}] (\tRF+.2,\yC)--(\tRF,\yC);
	
	\node(d) at (.1,\yA-.2){$\rvec$};
	\end{tikzpicture}
	\caption{Game at time $t_0$.}\label{fig:gameInit}
\end{subfigure}
\hspace{20 pt}
\begin{subfigure}{0.4\linewidth}
\begin{tikzpicture}[scale=1.7]
    \tikzmath{\yD = 0; \yA = .8; \tnu=.677; \tdelta=.1;
    \talpha=1/(1-\tnu^2); \tgamma=\talpha*\tnu; \tbeta=\tgamma*\tnu;
    \yC=\talpha*\yA-\tbeta*\yD; \tr=\yA-\yD; \tRC=\tgamma*\tr; \tRF=\tRC+\tdelta;}
    
    \draw[line width=0.3mm, black] (0,\yC) circle[radius=\tRF];
    \draw[black,thick,fill](0,\yC) circle[radius=1pt];
	\node(xF) at (0,\yC) [anchor=south west] {$\x_\mathcal{C}$};
    \tikzmath{\yF=\yC;}
    
    \tikzmath{\yD = 0.35; \yA = .85; \xD=.25; \xA=.3;
    \talpha=1/(1-\tnu^2); \tgamma=\talpha*\tnu; \tbeta=\tgamma*\tnu;
    \yC=\talpha*\yA-\tbeta*\yD; \xC=\talpha*\xA-\tbeta*\xD; \tr=sqrt((\yA-\yD)^2+(\xA-\xD)^2)); \tRC=\tgamma*\tr;}

	\draw[red,fill] (\xA,\yA) circle[radius=1pt];
	\draw[blue,fill](\xD,\yD) circle[radius=1pt];
	\draw[black,fill](\xC,\yC) circle[radius=1pt];
	\draw[white,fill] (0,0) circle [radius=1pt] node [anchor=west] {$\x_P$}; 
	
	\draw[red](\xC,\yC) circle[radius=\tRC];
	\draw[-{Latex[length=2mm]},dashed] (0,\yF)--(\xC,\yC);
	\node(xF) at (\xC,\yC) [anchor=west] {$\x_\mathcal{A}(t)$};
	\node(y) at (\xC-.2,\yC){$\y$};
	
	\tikzmath{\guessA=-35;}
	\draw (\xC,\yC)++(\guessA+180:\tRC) -- ++(\guessA+180:.85);
	\draw (\xC,\yC)++(\guessA:\tRC) -- ++(\guessA:.075);
	
	\node(dmin) at (1.1,.8){$d_{\min}$};
	\node(dmax) at (-.5,1.7){$d_{\max}$};
	\end{tikzpicture}
	\caption{Game at time $t>t_0$.}\label{fig:gameTwo}
\end{subfigure}

\caption{Depiction of game layout and construction of $\mathcal{C}$}
\label{fig:boundConstruction}
\end{figure}

At time $t$, let us denote the minimum and maximum distances between $\mathcal{A}(t)$ and $\partial\mathcal{C}$ by $d_{\min}(t)$ and $d_{\max}(t)$, respectively. Therefore,
\begin{align} \label{E:d-y}
    d_{\min}(t) &= R_\mathcal{C} - (R_{\mathcal{A}}(t) + \|\y(t)\|),\\
    d_{\max}(t) &= R_\mathcal{C} - (R_{\mathcal{A}}(t)  - \|\y(t)\|),\\
    \y(t) &= \x_\mathcal{A}(t) - \xf.
\end{align}

Showing that $\mathcal{A}(t)\in\mathcal{C}^\circ$ for all $t\ge t_0$ is equivalent to showing $d_{min}(t),d_{max}(t)>0$ for all $t>t_0$. We use the notation $\hat{\x}$ to denote a unit vector along the direction of the vector $\x$.

\begin{theorem} \label{Thm:pursuer_strategy}
The pursuit law
\begin{equation}\label{eq:purStrat}
    \vvec_P = \hat{\z}_P,
\end{equation}
\begin{equation} \label{eq:zd}
    \z_P = (R_\mathcal{C}-R_{\mathcal{A}}(t)) \rvechat(t) + \nu \y(t)
\end{equation}
ensures that 
\begin{enumerate}[label=(\roman*)]
    \item $\mathcal{A}(t)\in\mathcal{C}^\circ,\forall t>t_0$, and
    \item Capture is guaranteed in a finite time $t_{\rm capture} \le t_0 + 2(1+\nu^{-1})R_\mathcal{C}\ln (R_\mathcal{C}/\delta)$.
\end{enumerate}
\end{theorem}

\begin{proof}
\begin{enumerate}[wide, labelindent=0pt, label=(\roman*)]
    \item We consider the Lyapunov-like function $V(t) \triangleq d_{\max}(t)d_{\min}(t)$ and will show that $V(t) >0$ for all $t$.
In particular, $d_{\min}(t) > 0$ for all $t$, which is necessary and sufficient to conclude that $\mathcal{A}(t)$ remains within $\mathcal{C}^\circ$ for all time. 
To show this, we consider the time derivative of $V$ as follows:
\begin{align*}
    V(t) &= d_{\max}(t)d_{\min}(t) = (R_\mathcal{C}-R_{\mathcal{A}}(t))^2 -\|\y(t)\|^2\\
    \dot V(t) &= -2(R_\mathcal{C}-R_{\mathcal{A}}(t))\dot{R}_\mathcal{A}(t) -2\y(t)\T \dot{\y}(t)\\
    &=-2(R_\mathcal{C}-R_{\mathcal{A}}(t))\gamma \rvechat(t)\T \dot \rvec(t) - 2\y(t)\T \dot\x_\mathcal{A}(t)\\
    &= -2(R_\mathcal{C}-R_{\mathcal{A}}(t))\gamma \rvechat(t)\T(\vvec_E(t) -\vvec_P(t)) \\
    &\quad- 2\y(t)\T (\alpha \vvec_E(t) - \beta \vvec_P(t))\\
    &= 2\gamma\left[(R_\mathcal{C}-R_{\mathcal{A}}(t))\hat {\rvec}(t) + \nu \y(t) \right]\T \vvec_P(t) \\
    &\quad - 2\alpha\left[\nu (R_\mathcal{C}-R_{\mathcal{A}}(t))\hat {\rvec}(t) +  \y(t) \right]\T \vvec_E(t).
\end{align*}
For convenience, denote $\z_E(t)=\nu (R_\mathcal{C}-R_{\mathcal{A}}(t))\hat {\rvec}(t) +  \y(t)$. Now,
\begin{align*}
    \dot V(t) &= 2\gamma \z_P(t)\T \vvec_P(t) - 2\alpha \z_E(t)\T \vvec_E(t) \\
    &\ge 2\gamma \z_P(t)\T \vvec_P(t) - 2\alpha \|\z_E(t)\|\nu 
\end{align*}
where the inequality is obtained by noticing that the vector $\vvec_E(t)$ has a maximum magnitude of $\nu$ and using the Cauchy-Schwarz inequality. Note that this step does not require the evader to choose any particular policy; it is simply a worst-case analysis from the perspective of the pursuer.
Assuming for the moment that $\|\z_P(t)\|>0$ for all $t\ge t_0$, we can choose the pursuer strategy $\vvec_P = \hat{\z}_P(t)$, which gives us
\begin{align}
    \dot V(t) &\ge 2\gamma (\|\z_P(t)\| -\|\z_E(t)\|) = 2\gamma \frac{\|\z_P(t)\|^2 -\|\z_E(t)\|^2}{\|\z_P(t)\| + \|\z_E(t)\|} \nonumber\\
    &= 2\gamma \frac{(1-\nu^2) ((R_\mathcal{C}-R_{\mathcal{A}}(t))^2-\|\y(t)\|^2)}{\|\z_P(t)\| + \|\z_E(t)\|}\nonumber\\
    &= 2\nu\frac{V(t)}{\|\z_P(t)\| + \|\z_E(t)\|}.\label{eq:Vdot}
\end{align}
Therefore, as long as  $\|\z_P(t)\|>0$, the pursuer's policy is well-defined and $\dot V(t) \ge 0$. Hence, $V(t)$ is non-decreasing. Next, we will prove that $\|\z_P(t)\|>0$ for all $t\ge t_0$. We will proceed by contradiction.
First, note that $\|\z_P(t_0)\|=\delta>0$ since $\y(t_0)=0$ and $R_\mathcal{C}=R_\mathcal{A}(t_0)+\delta.$ 
Now, suppose that $\tau >t_0$ be the first time instance when $\|\z_P(\tau)\|=0$.
Therefore, from \eqref{eq:zd} 
\begin{align*}
    &(R_\mathcal{C}-R_\mathcal{A}(\tau))\rvechat(\tau) + \nu \y(\tau) = 0\\
    \implies & \rvechat\T(\tau) \left((R_\mathcal{C}-R_\mathcal{A}(\tau))\rvechat(\tau) + \nu \y(\tau)\right) = 0\\
  \implies &  R_\mathcal{C}-R_\mathcal{A}(\tau) + \nu \rvechat(\tau)\T \y(\tau) = 0\\
  \implies &  d_{\min}(\tau) + \|\y(\tau)\| + \nu \rvechat(\tau)\T \y(\tau) = 0\\
  \implies & d_{\min}(\tau) = - \|\y(\tau)\| - \nu \rvechat(\tau)\T \y(\tau) \\
  & \qquad \quad\le - \|\y(\tau)\| + \nu \|\y(\tau)\| \le 0. 
\end{align*}
The condition $d_{\min}(\tau) \le 0$ implies that there must have been a time in the interval $[t_0, \tau]$ when $d_{\min}$ becomes zero and thus $V$ becomes zero. 
This is not possible since $V(t_0)=\delta^2$ and $\dot V$ has been positive for the entire interval $[t_0,\tau)$ due to the hypothesis that $\|\z_P(t)\|>0$ for all $t \in [t_0,\tau)$.
This leads to a contradiction, and hence, $\|\z_P(t)\|$ cannot be zero for any time $t \ge t_0$.  \hfill $\qed$
\vspace{.5cm}
\item From (i), we have $R_\mathcal{C}\ge R_\mathcal{A}(t)$ for all $t\ge t_0$. Notice that $\|\z_P(t)\| \le (R_\mathcal{C}-R_{\mathcal{A}}(t))+\nu \|\y\| \le (1+\nu)R_\mathcal{C}$ and similarly $\|\z_E(t)\| \le \nu(R_\mathcal{C}-R_{\mathcal{A}}(t))+\|\y\| \le (1+\nu)R_\mathcal{C}$ for all $t$.
From \eqref{eq:Vdot}, we obtain
\begin{align*}
    &\dot V \ge 2\nu\frac{V(t)}{\|\z_P(t)\| + \|\z_E(t)\|} \ge \nu\frac{V(t)}{(1+\nu)R_\mathcal{C}}\\
    &V(t) \ge V(t_0)\exp{\frac{\nu(t-t_0)}{(1+\nu)R_\mathcal{C}}}.
\end{align*}
Since $V(t)=R_\mathcal{C}^2$ implies capture, the theorem holds.\hfill  \qed 
\end{enumerate}
\end{proof}

The bound on capture time in Theorem~\ref{Thm:pursuer_strategy}(ii) is a conservative bound, and its purpose is to show that the evader cannot delay the capture time indefinitely. 
The exact capture time, or a tighter bound thereof, can be obtained by considering an optimal  control problem. Theorem~\ref{Thm:pursuer_strategy}(i)'s spatial bound of capture within $\mathcal{C}$ can also be made tighter.

\begin{corr}
Although we inflate the initial AC radius by $\delta$ to construct the region $\mathcal{C}$, it is guaranteed that capture will occur within a $\delta_1 $ neighborhood of the initial AC, where $\delta_1 = R_\mathcal{C} - \sqrt{R_\mathcal{C}^2 -\delta^2} < \delta$. 
\end{corr}

\begin{proof}
At the final time, $d_{\min}(t_f) +d_{\max}(t_f)= 2R_\mathcal{C}$ and thus $V(t_f) = d_{\min}(t_f)d_{\max}(t_f)=d_{\min}(t_f)(2R_\mathcal{C}-d_{\min}(t_f))\ge V(t_0)=\delta^2$. Define $\delta_1$ to be the smallest possible $d_{\min}(t_f)$, and the corollary follows. \hfill $\qed$ 
\end{proof}

The pursuit strategy \eqref{eq:purStrat} guarantees that the distance between the pursuer and the evader at time $t>t_0$ is strictly smaller that their initial distance regardless of the evader's strategy.

\begin{lemma}\label{lm:distance}
For all $t > t_0$, $\|\rvec(t)\| < \|\rvec(t_0)\|$.
\end{lemma}

\begin{proof}Without loss of generality, we assume that capture has not happened by time $t$.
From the proof of Theorem~\ref{Thm:pursuer_strategy}, we notice that $V(t) > V(t_0)$ for all $t$. 
Therefore, $(R_\mathcal{C}-R_{\mathcal{A}}(t))^2-\|\y(t)\|^2 > (R_\mathcal{C}-R_\mathcal{A}(t_0))^2-\|\y(t_0)\|^2$. 
Given that $\y(t_0)=0$ and Theorem~\ref{Thm:pursuer_strategy} ensuring that $R_{\mathcal{A}}(t) < R_\mathcal{C}$ for all $t$, we obtain from the above inequality that $R_\mathcal{C}-R_{\mathcal{A}}(t) > R_\mathcal{C}- R_\mathcal{A}(t_0)$, or equivalently, $R_{\mathcal{A}}(t) < R_\mathcal{A}(t_0)$.
Then, the Lemma is obtained by using the fact that $R_{\mathcal{A}}(t) = \gamma \|\rvec(t)\|$ for all $t$.
\end{proof}

\begin{rem}
Since $\|\z_P(t_0)\|>0$, the pursuer policy~\eqref{eq:purStrat} is a continuous function of the state, is always well-defined, and is continuous with respect to time.
\end{rem}

This vital property ensures that no singular surface exists, sidestepping any of the issues described in Appendix~\ref{sec:twoGoal} and present in much of the prior literature.

\begin{corr}[Mobile Sensing Game]
In limited-information games, where the pursuer has a circular sensing distance $\rho\ge ||\rvec(t_0)||$, it is guaranteed to capture the evader using the strategy outlined in Theorem~\ref{Thm:pursuer_strategy} without losing sensing of the evader.
\end{corr}
The proof follows directly from Lemma~\ref{lm:distance}. 

\begin{rem}\label{rem:lowSpeed}
If the evader employs any speed below $\nu$,
then all the guarantees in this section hold.
\end{rem}

At first, this may seem like a trivial observation, since choosing speeds below $\nu$ are explicitly in the admissible set. However, we will discuss implications of this along with other implications of the main results in the next section.

\section{Implications}

This section provides corollaries, in the context of Games of Final Location and Target Guarding problems, that immediately follow from our main result.

\begin{definition}[Game of Final Location]
Consider a continuous value field $\phi: \reals^2 \mapsto \reals$ defined over the planar environment.
For the pursuer-winning scenario, we define the \emph{Game of Final Location} as the Game of Degree with the terminal payoff given as
\begin{equation}
    P_P(\vvec_P, \vvec_E) = -P_E(\vvec_P, \vvec_E) = \phi(\x_E(t_f)),
\end{equation}
where $t_f$ is the terminal time, and $\x_E(t_f)$ denotes the location where capture occurs. In other words, the pursuer seeks to maximize $\phi(\x_E(t_f))$, while the evader seeks to minimize it.
\end{definition}

Define the minimum payoff inside the initial AC as
\begin{equation}
    \phi^*\triangleq\min_{\x\in\mathcal{A}(t_0)}\phi(\x)
\end{equation}
and $\mathcal{X}^*\triangleq\argmin_{\x\in\mathcal{A}(t_0)}\phi(\x)$ as the set of corresponding locations. 

\begin{corr}\label{cor:epsNash}
For any $\varepsilon>0$, there is a set of strategies where the pursuer follows \eqref{eq:purStrat} and the evader moves directly to some $\x\in\mathcal{X}^*$ which is an $\varepsilon$-equilibrium strategy pair.
\end{corr}
\begin{proof}
First, recall that the evader is dominant for any point inside  $\mathcal{A}(t_0)$.
That is, without regard to the pursuer's strategy, the evader can set an open loop strategy to travel directly toward any point in $\mathcal{A}(t_0)$ at max speed, and it will reach the point before the pursuer can capture it. Particularly, the evader can reach any point in $\mathcal{X}^*\subset\mathcal{A}(t_0)$. Then, the evader can simply wait at such a point until capture. Therefore, $\phi(\x_E(t_f))\leq\phi^*$.

Conversely, Theorem~\ref{Thm:pursuer_strategy} implies that, regardless of evader strategy, capture will occur within $\mathcal{C}^\circ$. Due to the continuity property of $\phi$, for any $\varepsilon>0$, there always exists a $\delta>0$ such that the pursuer's strategy enforces $\phi(\x_E(t_f))\ge\min_{\x\in\mathcal{C}}\phi(\x)\ge\phi^*-\varepsilon$.
\hfill $\qed$
\end{proof}

The following Corollary establishes the Value of the Game of Final Location by first deriving Upper and Lower Value functions and then showing that they can be made to be equal~\cite{bansal2017hamilton,friedman1970definition}.

\begin{corr}\label{cor:gameOfLocation}
The Value of the Game of Final Location is $\phi^*$.
\end{corr}
\begin{proof}
The Upper Value of the game is
\begin{equation}
    \overline{V} \triangleq \prettyinf_{\vvec_E} \sup_{\vvec_P} P_P(\vvec_P, \vvec_E) = \phi^*,
\end{equation}
following directly from the first inequality in the proof of Corollary~\ref{cor:epsNash}. The Lower Value of the game is
\begin{equation}
    \underline{V} \triangleq \sup_{\vvec_P} \prettyinf_{\vvec_E} P_P(\vvec_P, \vvec_E) \ge \sup_{\vvec_P} (\phi^*-\varepsilon) = \phi^*.
\end{equation}
Note that while there is no single $\delta$-strategy for the pursuer to guarantee that a particular payoff is $\phi^*$, the Value is defined by considering the supremum of all possible strategies, including the limit as $\delta\rightarrow 0$.
Then, because $\underline{V} \leq \overline{V}$ by definition, the Value must be $\phi^*$. \hfill $\qed$
\end{proof}

The standard Game of Distance to Target is a special case of the Game of Final Location, where the field $\phi$ is given by the distance from a target set:
\begin{equation}
    \phi(\x) = \min_{\z\in\mathcal{T}}\|\x-\z\|.
\end{equation}

\noindent In some works, a Game of Distance to Target is constructed only to find a barrier for a Target Guarding Problem. In others, the Game of Distance to Target is, itself, the topic of interest. In either case, the traditional approach can be quite complicated for general target sets (for an example of an intricate target set, see Figure~\ref{fig:generalTargetGuarding}). Corollary~\ref{cor:gameOfLocation} solves the Game of Degree, and now we turn to the Game of Kind.

\begin{figure}[h]
	\centering
	\includegraphics[width=0.35\textwidth]{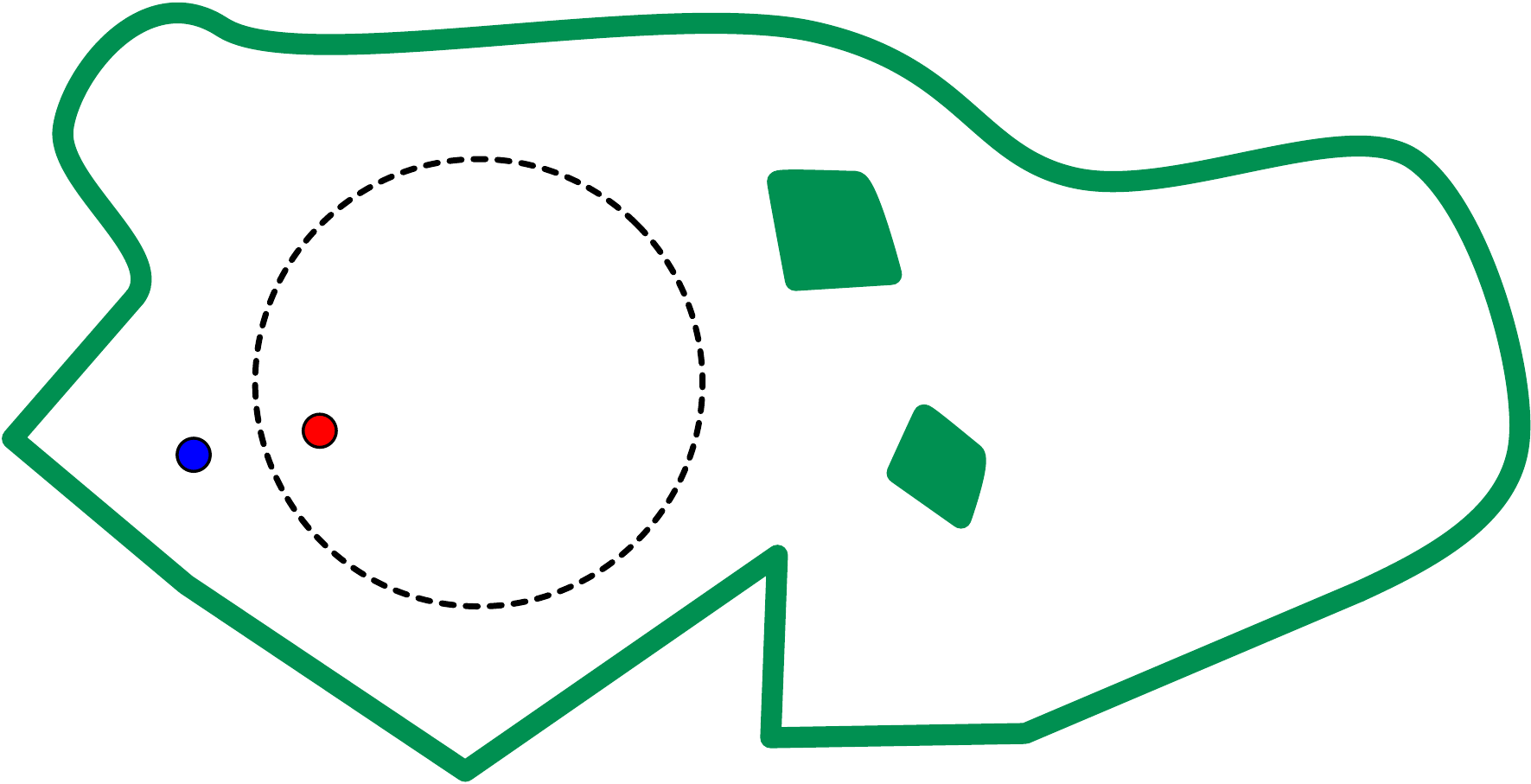}
	\caption{Generalized Target Guarding Problem (where green regions represent the target set). Note that the target set does not act as obstacles.}
	\label{fig:generalTargetGuarding}
\end{figure}

\begin{definition}[Generalized Target Guarding]
We define the \emph{Generalized Target Guarding Problem} as a Game of Kind involving an evader, pursuer, and a target set.
The target set $\mathcal{T}$ is a collection of points and surfaces fixed in space.
The evader wins the game by reaching any point in the target set without being captured by the pursuer, whereas the pursuer wins the game by capturing the evader before it reaches the target set.\footnote{For the purposes of this paper, ties go to the evader.}
\end{definition}

\begin{corr}
For a Generalized Target Guarding Problem, there exists $\delta>0$ such that the strategy presented in Theorem~\ref{Thm:pursuer_strategy} guarantees pursuer's win iff the initial AC does not contain any point in the target set.
\end{corr}
\begin{proof}
Necessity follows trivially from longstanding results that the initial AC forms the evader dominance region.

For sufficiency, let $d^*>0$ be the shortest distance between $\partial\mathcal{A}(t_0)$ and any point on the target set.
There can be multiple points on the target set that achieve this shortest distance.
Now we select $\delta=0.5d^*$ to inflate the initial AC and construct the circle $\mathcal{C}$, within which the capture is guaranteed.
Since $\mathcal{C}$ does not intersect with the target set, capture occurs outside of the target set, i.e., before evader reaches the target set. \hfill $\qed$
\end{proof}

\begin{corr}[Multi-pursuer Game]
Suppose there are multiple pursuers, possibly with different speeds, each implementing the strategy in Theorem~\ref{Thm:pursuer_strategy}. Denote the intersection of all the initial ACs as $\mathcal{I}$. Capture will occur arbitrarily close to $\mathcal{I}$, the Value of the Game of Final Location is similarly the maximum value of $\phi$ over $\mathcal{I}$, and the pursuers will win any Target Guarding Game if $\mathcal{I}$ does not contain any point in the target set.
\end{corr}
Note that the implementation of the strategy is completely decentralized, and there is no coordination required among the pursuers. In fact, each pursuer does not even need information about its fellow pursuers.

The exposition up to this point assumes that the pursuer has perfect knowledge about the evader's maximum speed $\nu$. 
We can also consider scenarios where the pursuer has some uncertainty in the evader's maximum speed. 
Denote the evader's true maximum speed by $\nut$, the pursuer's estimate of $\nut$ by $\Tilde{\nu}$, and with slight abuse of notation, denote $\mathcal{C}^o(\nu)$ to be the circular region defined for a speed ratio of $\nu$. 
It can be verified that $\mathcal{C}^o(\nu_1) \subset \mathcal{C}^o(\nu_2)$ if $\nu_2 > \nu_1$, and the next corollary follows trivially from Remark~\ref{rem:lowSpeed}.

\begin{corr}[Uncertain Maximum Speed] \label{corr:robust_strategy}
If $\Tilde{\nu} \ge \nut$, the pursuer strategy \eqref{eq:purStrat} guarantees that the evader does not reach any point outside $\mathcal{C}^o(\Tilde{\nu})$.
\end{corr}

The tighter the estimate $\Tilde{\nu}$ is, the smaller the guaranteed capture region $\mathcal{C}^o(\Tilde{\nu})$ will be. Therefore, while having a smaller value for $\Tilde{\nu}$ ensures capture within a smaller region, it also increases the risk of $\Tilde{\nu}$ being smaller than $\nut$ and consequently not capturing at all. 
On the other hand, a higher value of $\Tilde{\nu}$ reduces the risk of $\Tilde{\nu}$ being smaller than $\nut$ but it also increases the guaranteed capture region.
Thus, in context of the Game of Final Location, a natural trade-off between the risk of not capturing at all and the area of the guaranteed capture region (and thus, perhaps an increase in the final value) arises automatically through the pursuer's estimate $\Tilde{\nu}$. Furthermore, this simple result only considers the case where the pursuer makes one estimate at the beginning of the game and commits to a strategy based on that estimate. How/when the pursuer should update $\Tilde{\nu}$ (coupled with how the evader should structure his strategy to counter the pursuer's update law) remains an open question.

We finish with the Generalized Target Guarding version of Corollary~\ref{corr:robust_strategy}. With a now familiar abuse of notation, let $\mathcal{A}(t;\nu)$ denote the AC computed based on the positions at time $t$ using a particular value of $\nu$, whose true value may be unknown to the players.
\begin{corr}[Critical Speed]\label{corr:critical_speed}
There exists a pursuer strategy to win the Generalized Target Guarding problem if and only if $\nut<\nu_\text{crit}$, where $\nu_\text{crit}<1$ is the smallest $\nu$ such that
\begin{equation}
    \min_{\x\in \mathcal{A}(t_0; \nu)}\phi(\x)=0.
\end{equation}
\end{corr}

\begin{proof}
For any $\Delta>0$, the pursuer plays the game with the strategy from Theorem~\ref{Thm:pursuer_strategy}, assuming $\Tilde{\nu}=\nu_\text{crit}-\Delta$. Then $\delta=\left(\frac{\nu_\text{crit}}{(1-\nu_\text{crit})^2}-\frac{\Tilde{\nu}}{(1-\Tilde{\nu})^2}\right)\|\rvec(t_0)\|>0$ guarantees a win for $\nut\leq\nu_\text{crit}-\Delta$. 
Such $\delta$ always exists since we have $\delta\rightarrow 0$ with $\Delta\rightarrow 0$. \hfill $\qed$

\end{proof}

This has implications for the Generalized Target Guarding Game with uncertainty in the speed ratio. Suppose the pursuer does not know the maximum speed of the evader. By playing the game assuming $\Tilde{\nu}$, the pursuer can lose only when $\nut>\nu_\text{crit}-\Delta$. Importantly, since the evader is guaranteed to win anyway when $\nut>\nu_\text{crit}$, only an infinitesimal parameter regime, $\nut\in\left(\nu_\text{crit}-\Delta,\nu_\text{crit}\right]$ remains in question. Therefore, in contrast to the Game of Final Location, the pursuer can (and the authors believe, should) avoid worrying about actively estimating $\nut$ for Generalized Target Guarding in favor of simply choosing $\Tilde{\nu}$ arbitrarily close to $\nu_\text{crit}$. The pursuer gains no advantage by performing such an estimate. Furthermore, by not using an estimator, it prevents any adversarial action based on exploiting the estimator.

\section{Conclusion and Limitations}
In this work, we provide a sufficient strategy for the pursuer which guarantees that a slower evader will be captured within a neighborhood arbitrarily close to the Apollonius circle constructed at the beginning of the game. 
In contrast to existing strategies in the literature, our strategy eliminates the nuances of considering particular problem geometries separately and their resultant singular surfaces.

The results of this paper do not apply to cases where the agents have more complex dynamics (higher order, nonholonomic, state spaces with other geometries), in games with a non-zero capture radius, or where obstacles are present. This is because the AC no longer accurately captures the dominance region. Furthermore, when the evader is faster than the pursuer, i.e., $\nu>1$, it is easy to see that capture cannot be enforced with only one pursuer. Multiple pursuers will have to cooperate to guarantee capture.
Finally, other games, such as the Game of Capture Time, still require further treatment.

\bibliographystyle{unsrt}
\bibliography{refs}

\appendix

\section{Two target problem leading to oscillation and evasion}\label{sec:twoGoal}

\begin{figure}[h]
\begin{subfigure}{0.4\linewidth}
\begin{tikzpicture}[scale=1.5]
    \tikzmath{\yD = 0; \yA = .8; \tnu=.677; \tdelta=.1;
    \talpha=1/(1-\tnu^2); \tgamma=\talpha*\tnu; \tbeta=\tgamma*\tnu;
    \yC=\talpha*\yA-\tbeta*\yD; \tr=\yA-\yD; \tRC=\tgamma*\tr; \tRF=\tRC+\tdelta;}

	\draw[red,fill] (0,\yA) circle[radius=1pt] node [anchor=west] {$\x_E$};
	\draw[blue,fill](0,\yD) circle[radius=1pt] node [anchor=west] {$\x_P$};
	
	\draw[](0,\yC) circle[radius=\tRC];
	
	\tikzmath{\gx=\tRC+.3;}
	\draw[target,very thick](\gx,0)--(\gx,\yC+\tRC);
	\draw[target,very thick](-\gx,0)--(-\gx,\yC+\tRC);
	\draw[black,fill](-\tRC,\yC) circle[radius=1pt] node [anchor=south west] {$I_1$};
	\draw[black,fill](\tRC,\yC) circle[radius=1pt] node [anchor=south east] {$I_2$};
	\draw[](-\tRC,\yC)--(-\gx,\yC) node [anchor=south west] {$d_1$};
	\draw[](\tRC,\yC)--(\gx,\yC) node [anchor=south east] {$d_2$};
	
	\draw[-{Latex[length=2mm]},dashed,blue] (0,\yD)--(-\tRC,\yC);
	\draw[-{Latex[length=2mm]},dashed,blue] (0,\yD)--(\tRC,\yC);
	\draw[-{Latex[length=2mm]},dashed,red] (0,\yA)--(-\tRC,\yC);
	\draw[-{Latex[length=2mm]},dashed,red] (0,\yA)--(\tRC,\yC);
	
\end{tikzpicture}\caption{Two possible strategies.}\label{fig:twoStrat}
\end{subfigure}
\hspace{20 pt}
\begin{subfigure}{0.4\linewidth}
\begin{tikzpicture}[scale=1.5]
    \tikzmath{\yD = 0; \yA = .8; \tnu=.677; \tdelta=.1;
    \talpha=1/(1-\tnu^2); \tgamma=\talpha*\tnu; \tbeta=\tgamma*\tnu;
    \yC=\talpha*\yA-\tbeta*\yD; \tr=\yA-\yD; \tRC=\tgamma*\tr; \tRF=\tRC+\tdelta;}

	\draw[red,fill] (0,\yA) circle[radius=1pt];
	\draw[blue,fill](0,\yD) circle[radius=1pt];
	
	\draw[](0,\yC) circle[radius=\tRC];
	
	\tikzmath{\gx=\tRC+.3;}
	\draw[target,very thick](\gx,0)--(\gx,\yC+\tRC);
	\draw[target,very thick](-\gx,0)--(-\gx,\yC+\tRC);
	\draw[black,fill](-\tRC,\yC) circle[radius=1pt] node [anchor=south west] {$I_1$};
	\draw[black,fill](\tRC,\yC) circle[radius=1pt] node [anchor=south east] {$I_2$};
	\draw[](-\tRC,\yC)--(-\gx,\yC);
	\draw[](\tRC,\yC)--(\gx,\yC);
	
	\draw[-{Latex[length=2mm]},dashed,blue] (0,\yD)--(\tRC,\yC);
	\draw[-{Latex[length=2mm]},dashed] (0,\yD)--(0,\yA);
	\tikzmath{\a1=55;}
	\coordinate (E) at (0, \yA);
	\coordinate (P) at (0, \yD);
	\coordinate (I) at (\tRC, \yC);
	\pic ["$\psi$",-,draw,color=black,angle radius=18pt,angle eccentricity=1.5] {angle=I--P--E};
\end{tikzpicture}
\caption{Angle if P goes to $I_2$}\label{fig:twoGoalAngle}
\end{subfigure}

\caption{Two target problem on singular surface.}
\label{fig:twoGoal}
\end{figure}
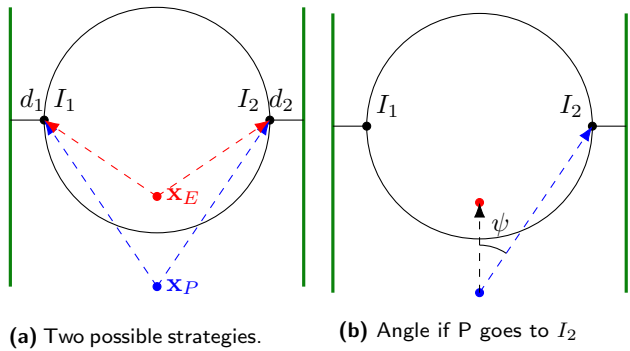

Similar to \cite{shishika2021partial}, consider an example where there are two parallel vertical target regions (instead of circular target regions), as shown in Figure~\ref{fig:twoGoal}. A typical differential game formulation may discover equilibrium strategies where both agents head directly to the point on the AC that is closest to one of the two targets (i.e., $J = \max_{\vvec_P} \min_{\vvec_E} \min \left(d_1, d_2\right)$). If $d_1<d_2$, then both agents move directly to $I_1$. Conversely, if $d_1>d_2$, then both agents move directly to $I_2$.
When both agents aim at the same point, the point remains fixed in space. With the cost defined as $\min(d_1, d_2)$, the surface where $d_1 = d_2$ is a dispersal surface -- the agents' equilibrium control inputs are undefined.
Crucially, the nonsingular strategies seem to be subgame perfect -- if $d_1<d_2$ and both agents choose to move toward $I_1$, then $d_1<d_2$ continues to hold for the duration of the game (similar for $d_1>d_2$) -- i.e., there is no perpetuated dilemma~\cite{isaacs1965differential}.
However, a critical question remains.

\begin{question}
What happens on the singular surface?
\end{question}

We will discuss three different possible answers.

\begin{answer}\label{ans:IMS}
Both agents choose an instantaneous mixed strategy.
\end{answer}

This is the natural first answer, especially considering the subgame perfect nature of the nonsingular strategies. Both agents pick either $I_1$ or $I_2$, say with 50\% probability for the first instant. Immediately thereafter, either $d_1<d_2$ or $d_1>d_2$, and the agents' equilibrium control actions for the remainder of the game are well-defined; this is true whether the agents choose the same aim point or not.

However, there is a problem with this reasoning. In a Nash equilibrium, no player has anything to gain by changing only their own strategy. We will show by contradiction that Answer~\ref{ans:IMS} is not actually a Nash equilibrium.

We can view the pursuer's strategy (head directly to $I_1$ or $I_2$, choosing an IMS between these options on the singular surface) as a bang-bang type strategy. That is, there are only two possible \textit{actions} he may choose to take. The question then is: holding the pursuer's strategy constant, is there anything else the evader may prefer to do? Indeed, there is. Consider what happens if the evader chooses to move directly up. The pursuer's strategy will chatter back and forth across the singular surface, and the system will remain at/near the singularity.

This behavior is similar to that of~\cite{pachter2019singular}, but can be much more dangerous. Consider the distance between the two agents $r=\lVert \mathbf{r}\rVert$. The time derivative of this distance is
\begin{equation*}
    \dot{r}=\mathbf{r}\cdot(\mathbf{v}_E-\mathbf{v}_P).
\end{equation*}
In this vertical arrangement on the singular surface, if the evader heads straight up, $\hat{\mathbf{r}}\approx\hat{\mathbf{v}}_E$. If the pursuer heads toward $I_2$, for example (shown in Figure~\ref{fig:twoGoalAngle}), then $\cos{\psi}\approx\frac{1}{\sqrt{\nu^2+1}}$ and

\begin{equation}
    \dot{r}\approx r\left(\nu-\frac{1}{\sqrt{\nu^2+1}}\right).\nonumber
\end{equation}

Notice that $\dot{r}$ becomes positive when $\nu$ is sufficiently large ($>\approx0.786$). In this case, the result of the pursuer control chattering is that the distance between the agents will always \textit{increase} (at least until the AC grows to intersect a target region, at which point, there exists an evader strategy to win). A video of a discretized simulation of this behavior can be viewed at \url{https://youtu.be/7J_8l_Uy6f8}.

This result is similar to that in Section 4 of \cite{vonmoll2019robust}, and as in that paper, a full treatment of this scenario requires solutions to differential equations with discontinuous right-hand sides (i.e. in the sense of Filippov~\cite{filippov2013differential}) due to the fast switching behavior of the pursuer. That said, supposing that the policy does result in a continuous solution, we have seen that the evader \textit{can} benefit from choosing a different strategy (go straight up instead of toward $I_1$ or $I_2$), and so Answer~\ref{ans:IMS} cannot be a Nash equilibrium.

\begin{answer}
The pursuer heads straight toward the evader.
\end{answer}

The analysis for this answer likely follows closer to~\cite{pachter2019singular}, and it resembles adding a form of `deadzone' into the pursuer's policy. We assume that the pursuer commits to heading straight toward the evader for some finite time $T$ once the state is on or sufficiently close to the singular surface. This commitment may allow the pursuer to alleviate the concern of Answer~\ref{ans:IMS} at the cost of some penalty in the reward function for the period of time that the pursuer is heading straight toward the evader ($J$ can be decreasing during this time). $T$ may need to be calibrated according to $\nu$.

Further, unlike \cite{pachter2019singular}, a full analysis of this answer would require us to (after specifying the pursuer's committed strategy) determine what the best evader counter-strategy is. Should the evader simply harvest the small reward gain and then proceed to a critical point? Can it instead stay near enough to the singular surface to occasionally (and advantageously) trigger the pursuer's deadzone? These questions would require significant effort, likely repeated for a plethora of different geometries, and it is clear why this paper's main result is to be preferred.

\begin{answer}
A viscosity solution.
\end{answer}

There has been great historical effort in describing viscosity solutions to differential games on singular surfaces and connecting them to dynamic programming problems, such that they can be estimated using numerical methods~\cite{falcone2006viscosity}. However, these methods are still difficult in practice for many problems of interest~\cite{lafflito2017viscosity}, and even if solved, generally only provide numerical results. The authors believe that these difficulties are easily overcome for many problems by using the result of this paper instead.

\end{document}